\documentclass[12pt]{amsart}

\usepackage{amssymb}
\usepackage{graphicx}

\usepackage{lineno}

\usepackage{color}

\newtheorem{theorem}{Theorem}
\newtheorem{lemma}[theorem]{Lemma}
\newtheorem{proposition}[theorem]{Proposition}
\newtheorem{corollary}[theorem]{Corollary}
\newtheorem{conjecture}[theorem]{Conjecture}
\newtheorem{problem}{Problem}
\newtheorem{question}{Question}

\newtheorem{them}{Theorem}
\newtheorem{lema}[them]{Lemma}

\theoremstyle{definition}

\theoremstyle{remark}
\newtheorem{remark}[theorem]{Remark}

\begin{document}

\title[The bondage number of graphs: S-vertices]{ The bondage number of graphs on topological surfaces: 
degree-S vertices and the average degree}
\author{Vladimir Samodivkin}
\address{Department of Mathematics, UACEG, Sofia, Bulgaria}
\email{vl.samodivkin@gmail.com}
\today

\begin{abstract}
The bondage number $b(G)$ of a graph $G$ is the smallest number 
of edges whose removal from $G$ results in a graph with larger 
domination number. 
An orientable surface $\mathbb{S}_h$ of genus $h$, $h \geq 0$, is
obtained from the sphere $\mathbb{S}_0$ by adding $h$ handles. 
A non-orientable surface $\mathbb{N}_q$ of genus $q$, $q \geq 1$, 
 is obtained from the sphere by adding $q$ crosscaps. 
 The Euler characteristic of a surface is defined by
$\chi(\mathbb{S}_h) = 2 - 2h$ and $\chi(\mathbb{S}_q)= 2-q$. 
Let $G$ be a connected graph of order $n$ which is  $2$-cell embedded 
on a surface $\mathbb{M}$ with $\chi(\mathbb{M})= \chi$. 
We prove that $b(G) \leq 7+i$ when $\mathbb{M} = \mathbb{N}_i$, $i=1,2,3$, and 
$b(G) \leq 12$ when $\mathbb{M} \in \{\mathbb{N}_4, \mathbb{S}_2\}$. 
We give new arguments that improve the known upper bounds on the bondage number
at least   when $-7\chi/(\delta(G) - 5) < n \leq -12\chi$, $\delta(G) \geq 6$,
 where $\delta(G)$ is the minimum degree of $G$. 
  We obtain sufficient conditions for the validity of the inequality $b(G) \leq 2s-2$,  
provided $G$ has degree $s$ vertices. 
In particular, we prove that if  $\delta (G) = \delta \geq 6$, $\chi  \leq -1$  
and $-14\chi  < \delta - 4 + 2(\delta -5)n$ then  $b(G) \leq 2\delta -2$.  
We show that if  $\gamma (G) = \gamma \not = 2$, where $\gamma (G)$ 
is the domination number of $G$, then 
$n \geq \gamma + (1 + \sqrt{9+8\gamma-8\chi})/2$; 
the bound is tight. 
We also present upper bounds for the bondage number of graphs 
in terms of  the girth, domination number and Euler characteristic.  
As a corollary we prove that if $\gamma(G) \geq 4$ and $\chi \leq -1$, then 
$b(G) \leq 11 - 24\chi/(9 + \sqrt{41 - 8\chi})$.
 Several unanswered questions  are posed. 
\end{abstract}

\maketitle


 \linenumbers

\section{Introduction}
An orientable compact 2-manifold $\mathbb{S}_h$ or orientable surface $\mathbb{S}_h$ (see \cite{Ringel}) of genus $h$ is
obtained from the sphere by adding $h$ handles. Correspondingly, a non-orientable compact
2-manifold $\mathbb{N}_q$ or non-orientable surface $\mathbb{N}_q$ of genus $q$ is obtained from the sphere by
adding $q$ crosscaps. Compact 2-manifolds are called simply surfaces throughout the paper. 
 The Euler characteristic is defined by
$\chi(\mathbb{S}_h) = 2 - 2h$, $h \geq 0$,  and $\chi(\mathbb{N}_q ) = 2 - q$, $q \geq 1$.
 The Euclidean plane $\mathbb{S}_0$, the projective plane $\mathbb{N}_1$, 
  the torus $\mathbb{S}_1$, and the Klein bottle $\mathbb{N}_2$ are 
  all the surfaces of nonnegative Euler characteristic.

We shall consider graphs without loops and multiple edges. 
A graph $G$ is embeddable on a topological surface $\mathbb{M}$ if it admits a drawing on the
surface with no crossing edges. Such a drawing of $G$ on the surface $\mathbb{M}$ is called an
embedding of $G$ on $\mathbb{M}$.  
If a graph $G$ is embedded in a surface $\mathbb{M}$ then the connected components of $\mathbb{M} - G$
are called the faces of $G$.  
For such a graph $G$, we denote its vertex set,
edge set, face set, maximum degree, and minimum degree by $V (G)$, $E(G)$, $F(G)$, $\Delta(G)$, and
$\delta(G)$, respectively. Set $|G| = |V(G)|$, $\|G\| = |E(G)|$, and $f (G) = |F(G)|$. 
An embedding of a graph $G$ on a surface $\mathbb{M}$ is said to be $2$-cell 
if every face of the embedding is homeomorphic to an open disc.
 The   Euler's inequality  states
\begin{equation} \label{eq:euler}
|G| - \|G\| + f(G) \geq \chi(\mathbb{M})
\end{equation}
 for any graph $G$ that is embedded in $\mathbb{M}$. 
 Equality holds if $G$ is $2$-cell embedded in $\mathbb{M}$. 
By the genus $h$ (the non-orientable genus $q$) of a graph $G$
 we mean the smallest integer $h$ ($q$) such that $G$ has 
 an embedding into $\mathbb{S}_h$ ($\mathbb{N}_q$, respectively). 

 The girth of a graph $G$, denoted as $g(G)$, is the length of a shortest
cycle in $G$; if $G$ is a forest then $g(G) = \infty$. 
For any vertex $x$ of a graph $G$,  $N_G(x)$ denotes the set of all  neighbors of $x$ in $G$,  $N_G[x] = N_G(x) \cup \{x\}$ and the degree of $x$ is $d_G(x) = |N_G(x)|$. 
For a subset $A \subseteq V (G)$, let $N_G(A) = \cup_{x \in A} N_G(x)$, $N_G[A] = N_G(A) \cup A$, 
 and $\left\langle A, G\right\rangle$ be the subgraph of $G$ induced by $A$.
The distance between two vertices $x, y \in V(G)$ is  denoted by $d_G(x, y)$. 
 The average degree $ad(G)$ of a graph $G$ is defined as $ad(G) = 2\|G\| / |G|$. 

An independent set  is a set of vertices in a graph, no two of which are adjacent.  
The independence number $\beta_0(G)$ of a graph $G$ 
is the size of the largest independent set in $G$.
A dominating set for a graph $G$ is a subset $D\subseteq V(G)$ of 
vertices such that every vertex not in $D$ is adjacent to
at least one vertex in $D$. The minimum cardinality of a dominating 
set is called the domination number of $G$ and is denoted by $\gamma (G)$.  
The concept of domination in graphs has many applications 
in a wide range of areas within the natural and social sciences.
It is of practical significance to consider the vulnerability of domination 
in a communication network under link failure. 
 This can be measured
by the bondage number $b(G)$ of the underlined graph $G$,
defined in \cite{BHNS,FJKR} as the smallest number of edges whose 
removal from $G$ results in a graph with larger domination number. 
We refer the reader to \cite{jmx} for a detailed survey on this topic.
In general it is $NP$-hard to determine the bondage number 
(see Hu and Xu~\cite{HuXu}), and thus useful to find bounds for it. 

The main outstanding conjecture on the bondage number is the following:
\begin{conjecture}[Teschner \cite{Teschner}] \label{con1}
For any graph $G$, $b(G) \leq \frac{3}{2}\Delta (G)$. 
\end{conjecture}
Hartnell and Rall \cite{hr1} and Teschner \cite{Teschner2} showed that for 
the Cartesian product $G_n = K_n \times K_n$, $n \geq 2$, the bound of
Conjecture 1 is sharp, i.e. $b(G_n) = \frac{3}{2}\Delta(G_n)$. 
Teschner \cite{Teschner} also proved that Conjecture \ref{con1} holds when 
the domination number of $G$ is not more than $3$.

The study of the bondage number of graphs, 
which are $2$-cell embeddable  on a surface having 
negative Euler characteristic was initiated  
by Gagarin and Zverovich \cite{GagarinZverovich1} and is 
continued by the same authors in \cite{GagarinZverovich2}, 
Jia Huang in \cite{Jia Huang} and the present author in \cite{samcmj}.  
All these authors obtain upper bounds for the bondage number in terms of 
 maximum degree and/or   orientable and non-orientable genus of a graph. 
 In \cite{samajc}, the present author 
   gives  upper bounds for the bondage number in terms of   order, 
    girth and Euler characteristic of a graph.  
By Theorem 10 (ii) \cite{GagarinZverovich2} or by Theorem \ref{SGZ} (ii) below,  it immediately follows 
that Conjecture \ref{con1} is true for any graph $G$ such that all the following is valid: 
(a) $G$ is $2$-cell embeddable in a surface $\mathbb{M}$ with $\chi(\mathbb{M}) < 0$, 
 (b) $|G| > -12\chi(\mathbb{M})$, and (c) $\Delta (G) \geq 8$.  

In this paper we concentrate mainly on the case when a graph $G$ 
is $2$-cell embeddable in a surface $\mathbb{M}$ and $|G| \leq -12\chi(\mathbb{M})$. 
The rest of the paper is organized as follows. Section $2$ contains preliminary results. 
In section $3$ we give new arguments that improve the known upper bounds on the bondage number
at least   when $-7\chi(\mathbb{M})/(\delta(G) - 5) < |G| \leq -12\chi(\mathbb{M})$, $\delta(G) \geq 6$.
 We propose a new type of upper bound on the bondage number of a graph.
Namely we obtain sufficient conditions for the validity of the inequality $b(G) \leq 2s-2$,  
where $G$ is a graph having degree $s$ vertices, $s \geq 5$. 
In particular, we prove that if a connected graph $G$ 
is $2$-cell embeddable in an orientable/non-orientable surface $\mathbb{M}$ 
with negative Euler characteristic then  $b(G) \leq 2\delta -2$ 
whenever $-14\chi (\mathbb{M}) < \delta(G) - 4 + 2(\delta(G) -5)|G|$ and $\delta(G) \geq 6$.
 We also improve the known upper bounds for $b(G)$ when a graph $G$ is embeddable on
at least one of $\mathbb{N}_1, \mathbb{N}_2, \mathbb{N}_3, \mathbb{N}_4$ and $\mathbb{S}_2$. 
In section $4$ we give tight lower bounds for the number of vertices of graphs 
in terms of Euler characteristic and the domination number. 
We also present upper bounds for the bondage number of graphs 
in terms of the girth, domination number and  Euler characteristic.  
As a corollary, in section $5$ we give stronger than the known 
 constant upper bounds for the bondage number of graphs having  
 domination number at least $4$.

\section{Preliminary results}
We need the following notations and definitions. 

$\bullet$\ $V_{\leq r}(G) = \{ x \in V(G)\mid d_G(x) \leq r \}, r \geq 1$,

$\bullet$\ $V_{r}(G) = \{ x \in V(G) \mid d_G(x) = r \}, r \geq 1$,
%$\bullet$\ $V_{\geq i}(G) = \{ x \in V(G)\mid d_G(x) \geq i \}, i \geq 1$. 

$\bullet$\  $b_1(G) = \min\{d_G(x) + d_G (y) - 1 \mid x, y \in V(G)\  \mbox{and}\  1 \leq d_G(x,y)\leq 2\},$

$\bullet$\ $b_2(G) = \min_{x,y \in V(G)}\{d_G(x) + d_G(y) - 1 - |N_G(x) \cap N_G(y)| \mid xy \in E(G)\}$, 
\newpage
$\bullet$\ $b_3(G) = \min_{x,y \in V(G)} \{\max\{d_G(x) + d_G(y) - 1 - |N_G(x) \cap N_G(y)|, \ d_G(x)$

\hspace{5cm} $ + d_G(y) - 3\}\mid xy \in E(G)\}$, 

$\bullet$\ \cite{JHJMH}\ $B(G) = \min \{b_1(G), b_2(G)\}$,

$\bullet$\ $B^{\prime}(G) = \min \{b_1(G), b_3(G)\}.$

\begin{them} \label{d2}  If $G$ is a nontrivial graph, then 
\begin{itemize}
\item[(i)]  {\rm (Hartnell and Rall~\cite{hr2})} \ $b(G) \leq b_1(G) \leq 2ad(G) -1$;
 \item[(ii)] {\rm (Hartnell and Rall~\cite{hr1})} \ $b(G) \leq b_2(G)$. 
  \end{itemize}
\end{them}
By Theorem \ref{d2} and the above definitions we have $b_2(G) \leq b_3 (G)$ and 
\begin{equation} \label{eq:bprime}
b(G) \leq B(G) \leq B^{\prime}(G) \leq b_1(G)  \leq 2ad(G) -1.
\end{equation}
Note that, if a graph $G$ has no triangles then $B(G) = B^{\prime}(G) = b_1(G)$.
\begin{them}[Samodivkin ~\cite{samajc}]\label{SGZ}
Let $G$ be a connected graph embeddable on a surface $\mathbb{M}$ whose Euler
 characteristic $\chi$ is as large as possible and let $g(G) = g$.
 If $\chi \leq -1$ then:
 \begin{itemize}
 \item[(i)]  $ad(G) \leq \frac{2g}{g-2}(1 - \frac{\chi}{|G|})$;
 \item[(ii)]  $b(G) \leq 2ad(G) -1 \leq 3 + \frac{8}{g-2} - \frac{4\chi g}{|G|(g-2)}$.
 \end{itemize}
\end{them}
The same upper bound for $b(G)$, in case when $g \in \{3,4\}$, is obtained by Gagarin and Zverovich \cite{GagarinZverovich1}. 
\begin{them}[Gagarin and Zverovich ~\cite{GagarinZverovich2}]\label{GZ11}
Let $G$ be a connected graph  $2$-cell embedded  in a surface $\mathbb{M}$ 
 with  $\chi(\mathbb{M}) = \chi \leq -1$. Then 
\[
 b(G) \leq 2ad (G) - 1 \leq 11 + \frac{3\chi(\sqrt{17-8\chi} - 3)}{\chi - 1}. 
 \]
\end{them}
\begin{them}[Samodivkin ~\cite{samcmj}]\label{samczech}
Let $G$ be a connected  toroidal or Klein bottle graph. 
Then $b_2(G) \leq  \Delta(G)+3$ with equality if and only if one
of the following conditions is valid:
\begin{itemize}
\item[(P3)] $G$ is $4$-regular without triangles;
\item[(P4)] $G$ is $6$-regular and no edge of $G$ belongs to at least $3$ triangles.
\end{itemize} 
\end{them}

In \cite{FH}, Frucht and Harary define the corona of two graphs $G_1$ and $G_2$ 
to be the graph $G=G_1\circ G_2$ formed from one
copy of $G_1$ and $|G_1|$ copies of $G_2$, where the $i$th vertex of $G_1$ 
is adjacent to every vertex in the $i$th copy of $G_2$.
\begin{them}[Carlson and Develin ~\cite{CarlsonDevelin}]\label{CaDe}
Let $G$ be a graph of the form $G = H \circ K_1$. Then $b(G) = \delta (H) + 1$.
\end{them}

\begin{lema}[Sachs \cite{Sachs}, pp. 226-227] \label{fivedegree}
Let $G$ be a connected graph embeddable in a surface $\mathbb{M}$. 
If $\mathbb{M} \in \{\mathbb{S}_0, \mathbb{N}_1\}$ then $\delta (G) \leq 5$. 
If $\chi(\mathbb{M}) \leq 1$ then  
$\delta (G) \leq \left\lfloor (5+\sqrt{49-24\chi(\mathbb{M})})/2\right\rfloor$. 
\end{lema}

\begin{lemma} \label{edge}
Let $G$ be a graph embedded in a surface $\mathbb{M}$.
If $g(G) =g < \infty$ then $\|G\| \leq (|G| - \chi(\mathbb{M}))\frac{g}{g-2}$. 
\end{lemma}

\begin{proof}[Proof of Lemma \ref{edge}] 
 {\em Case} 1: A graph $G$ is connected. 
 Then there is a surface  $\mathbb{M}_1$ on which $G$ can be $2$-cell embedded. 
 Since clearly $gf(G) \leq 2\|G\|$, by \eqref{eq:euler} we have 
 $\chi(\mathbb{M}) \leq \chi(\mathbb{M}_1) = |G| - \|G\| +f(G) \leq |G| - \|G\| + \frac{2}{g}\|G\|$,  
 and the result easily follows.
  
 {\em Case} 2: A graph $G$ is disconnected. Then there is a connected supergraph $G_1$ for $G$
  such that (a) $V(G_1) = V(G)$ and  $E(G) \subsetneq E(G_1)$, and 
  (b) $G_1$ can be embedded in $\mathbb{M}$.  
 By Case 1 we immediately have 
 $\|G\| < \|G_1\| \leq (|G_1| - \chi(\mathbb{M}))\frac{g}{g-2}$.  
\end{proof}
The next lemma is  fairly  obvious and hence we omit the proof. 
\begin{lemma} [J. van den Heuvel \cite{jvdh}] \label{2connectednew}
Let $G$ be a connected graph $2$-cell embedded in a surface 
$\mathbb{M} \in \{\mathbb{S}_h, \mathbb{N}_q\}$, $v \in V(G)$ and $d_G(v) \geq 2$. 
Let $E_v = \{ xy \mid x,y \in N_G(v), x \not = y, xy \not \in E(G)\}$. 
Then there is a subset $D \subseteq E_v$, such that the graph $H = G+D$ is still 
$2$-cell embedded in $\mathbb{M}$ and 
\begin{itemize} 
\item[(i)] $\left\langle N_H(v), H\right\rangle$ is connected;
\item[(ii)] $\left\langle N_H(v), H\right\rangle$ is %two-connected 
 hamiltonian when $d_G(v) \geq 3$.
\end{itemize}
\end{lemma}

\section{Upper bounds: degree $s$ vertices}\label{sver}
\subsection{Results}
The main result of this section is the following theorem. 

\begin{theorem} \label{bBprim}
Let $G$ be a connected graph  $2$-cell embedded in $\mathbb{M}\in \{\mathbb{S}_p, \mathbb{N}_q\}$. 
 If $V_s(G) \not = \emptyset$ for some  $s \geq 4$ and 
\[
 -14\chi(\mathbb{M})  <  (s-4)\beta_0 (\left\langle V_s, G\right\rangle) + 2(s-5)|G| + 4|V_{\leq 2}| + 2\sum_{j=3}^{s-1}(5-j)|V_j|
\]
then $b(G)\leq B^{\prime}(G) \leq 2s-2$.
\end{theorem}
The next two corollaries immediately follow from Theorem \ref{bBprim}. 
\begin{corollary} \label{pet}
Let $G$ be a connected graph  $2$-cell embedded in $\mathbb{M}\in \{\mathbb{S}_p, \mathbb{N}_q\}$.  
 \begin{itemize}
 \item[(i)] If $V_5(G) \not = \emptyset$ and 
 $-14\chi(\mathbb{M})  < 4|V_{\leq 3}| + 2|V_{4}| + \beta_0 (\left\langle V_5, G\right\rangle)$ 
 then $b(G)\leq B^{\prime}(G) \leq 8$. 
 \item[(ii)] If $V_6(G) \not = \emptyset$ and 
 $-7\chi(\mathbb{M})  < 2|V_{\leq 3}| + |V_{4}| + \beta_0 (\left\langle V_6, G\right\rangle) + |G|$ 
 then $b(G)\leq B^{\prime}(G) \leq 10$.
  \item[(iii)] If $V_6(G) = \emptyset$, $V_7(G) \not = \emptyset$ and 
 $-14\chi(\mathbb{M})  < 4|V_{\leq 3}| + 2|V_{4}| + 3\beta_0 (\left\langle V_7, G\right\rangle) + 4|G|$ 
 then $b(G)\leq B^{\prime}(G) \leq 12$.
 \end{itemize}
 \end{corollary}
  This solves Conjecture \ref{con1} when (a) $G$ is as in Corollary \ref{pet}(i) and $\Delta (G) \geq 6$ 
   or (b) $G$ is as in Corollary \ref{pet}(ii) and $\Delta (G) \geq 7$ . 
\begin{corollary}\label{delta}
Let $G$ be a connected graph  $2$-cell embedded in $\mathbb{M}\in \{\mathbb{S}_p, \mathbb{N}_q\}$,  
 $\delta (G) = \delta \geq 4$ and 
$-14\chi(\mathbb{M})  < (\delta -4)\beta_0 (\left\langle V_\delta, G\right\rangle) + 2(\delta -5)|G|$. 
 Then $b(G)\leq B^{\prime}(G) \leq 2\delta -2$.
\end{corollary}
Hence we may conclude that Conjecture \ref{con1} is true whenever $G$ is as in Corollary \ref{delta} and 
  $4\delta(G) -4 \leq   3\Delta(G)$.  
  \begin{remark}\label{rem1}
 Let $G$ be a connected graph  $2$-cell embedded in a surface  
 $\mathbb{M}$ with $\chi(\mathbb{M}) = \chi \leq -1$ and  let  
 $\delta (G) = \delta \geq 6$. It is not hard to see that  
 if $-\frac{7\chi}{\delta - 5} - \frac{(\delta - 4)\beta_0 (\left\langle V_{\delta}, G\right\rangle)}{2(\delta - 5)} 
  < |G| \leq -12\chi$ then 
 the bound stated in Corollary \ref{delta}  is better than that given in Theorem \ref{SGZ}(ii).  
  \end{remark}
\begin{theorem}\label{deltamax}
Let $G$ be a connected graph embeddable on a surface $\mathbb{M}$ whose Euler
characteristic $\chi (\mathbb{M})$ is as large as possible. 
Let $G$ have no vertices of degree 
$\delta_{max}^{\mathbb{M}} = \max\{\delta (H) \mid \mbox{a graph\ } H \mbox{is\ } 
2\mbox{-cell\ }$ $\mbox{embedded in\ }  \mathbb{M}\}$. 
Then (a)  
$b(G) \leq B^{\prime}(G) \leq 2\delta_{max}^{\mathbb{M}} - 3$, and 
(b) if $\chi (\mathbb{M}) \leq 1$ then 
$b(G) \leq B^{\prime}(G) \leq 2\left\lfloor (5+\sqrt{49-24\chi(\mathbb{M})})/2\right\rfloor - 3$. 
\end{theorem}
There are infinitely many planar graphs $G$ without degree $\delta_{max}^{\mathbb{S}_0} =5$ vertices 
for which $B^{\prime}(G) = 2\delta_{max}^{\mathbb{S}_0} - 3 =7$. 
One such a graph is depicted in Figure  \ref{fig:bez5}.
 Notice that for a planar graph $G$ without degree $5$ vertices, the inequalities  
$b(G) \leq 7$ and $B(G) \leq 7$ due to  Kang and  Yuan \cite{KangYuan} and Huang and Xu \cite{JHJMH}, respectively. 

\begin{figure} \label{dmax}
	\centering
		\includegraphics{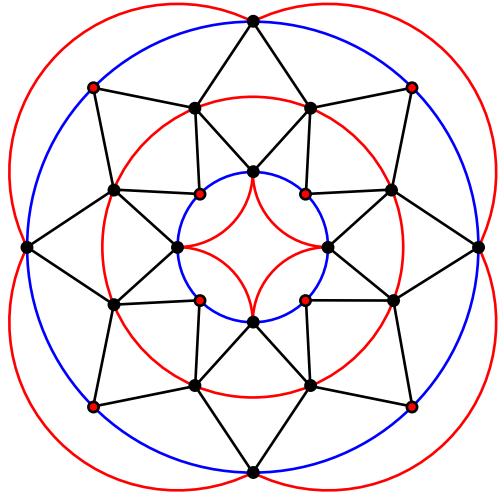}
	\caption{A planar graph $F$ (without degree $5$-vertices) having $B(F) = B^{\prime} (F) =7$}
	\label{fig:bez5}
\end{figure}
By Theorem \ref{deltamax}  and Corollary \ref{pet}(i) it immediately follows:
\begin{corollary} \label{S0N1}
If $G$ is $2$-cell embedded in $\mathbb{M}\in \{\mathbb{S}_0, \mathbb{N}_1\}$ 
           then $b(G) \leq B^{\prime}(G) \leq 8$. 
\end{corollary}
\begin{figure}[h]
	\centering
		\includegraphics{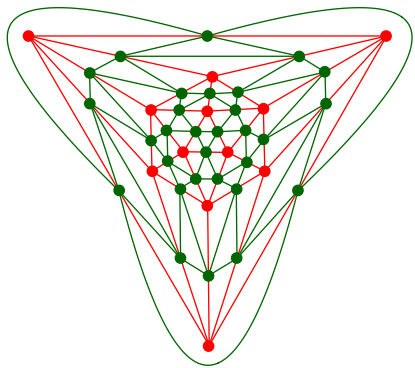}
	\caption{A planar triangulation $H$ with $B^{\prime}(H) = 8$}
	\label{fig:greenmaroon}
\end{figure}
The inequalities  $b(G) \leq 8$ and $B(G) \leq 8$ for planar graphs,  was proven by 
  Kang and  Yuan \cite{KangYuan} and Huang and Xu \cite{JHJMH}, respectively. 
Consider the planar graph $H$ shown in Figure \ref{fig:greenmaroon}  
(this graph is taken from \cite{JHJMH}).
 Each edge of $H$ belongs to exactly $2$ triangles, $\delta (H) = 5$, $\Delta (H) = 6$ and 
 all neighbors of any degree $5$ (red) vertex are degree $6$ (green) vertices.  
 This implies $B(H) = B^{\prime}(H) = 8$. Hence the upper bound for $B^{\prime}(G)$ in Corollary \ref{S0N1} is tight 
  when $\mathbb{M}= \mathbb{S}_0$.    
  
 Carlson and  Develin \cite{CarlsonDevelin} showed  that 
there exist planar graphs with bondage number $6$.
It is not known whether there is a planar graph $G$ with $b(G) \in \{7,8\}$. 

Consider the projective-planar graph $R$ depicted in Figure \ref{fig:Proj88}. 
Note that $R$ is a triangulation, each edge of $R$ is in exactly $2$ triangles, $\delta (R) = 5$,  
 there is no adjacent  degree $5$ (red) vertices and there is a degree $5$ vertex 
 adjacent to a degree $6$ (black) vertex.   
 This implies $B(R) = B^{\prime}(R) = 8$. Hence the upper bound for $B^{\prime}(G)$ in Corollary \ref{S0N1} is tight 
 when $\mathbb{M}= \mathbb{N}_1$. 
  Note that in the case when $\mathbb{M} = \mathbb{N}_1$, our result  
 is better than $b(G) \leq 10$ which was recently and independently obtained by Gagarin and Zverovich \cite{GagarinZverovich2} and %independently 
 by the present author \cite{samajc}. 
\begin{figure}[htbp]
	\centering
		\includegraphics{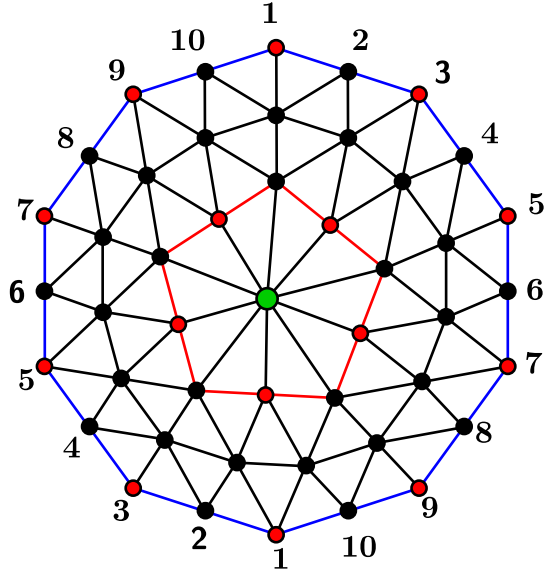}
	\caption{A projective-planar triangulation $R$ with $B^{\prime}(R) = 8$}
	\label{fig:Proj88}
\end{figure}

 It is well known that the non-orientable genus of $K_6$ is  $1$  \cite{Ringel}. 
Hence by Theorem \ref{CaDe} we obtain:
\begin{proposition}  \label{pro6}
There  exist projective-planar graphs with  bondage number $6$. 
In particular, $b(K_6 \circ K_1) =6$.
\end{proposition}
Corollary \ref{S0N1} and Proposition \ref{pro6} show that the maximum value of the bondage number 
 of projective-planar graph is $6,7$ or $8$. 
 %The following question naturally arises.
\begin{question}  \label{q1}
%Are there projective-planar graph $G$ with  $B(G)$ =8. 
Is there a projective-planar graph $G$ with  $b(G)  \in \{7,8\}$? 
\end{question}
In the next corollary we improve the known upper bound for 
the bondage number of Klein bottle graphs from $11$ (Gagarin and Zverovich \cite{GagarinZverovich2}) to $9$.
 \begin{corollary} \label{S1N2}
Let $G$ be $2$-cell embedded in $\mathbb{M}\in \{\mathbb{S}_1, \mathbb{N}_2\}$. 
            Then $b(G) \leq B^{\prime}(G) \leq 9$. 
            Moreover,  $B^{\prime}(G)=9$ if and only if $G$ is a 
            $6$-regular triangulation in $\mathbb{M}$.
            % and each edge of $G$ belongs to exactly two triangles of $G$.
\end{corollary}
It is an immediate consequence of Euler's formula that any $6$-regular 
graph embedded in $\mathbb{M}\in \{\mathbb{S}_1, \mathbb{N}_2\}$  is a triangulation. 
Altshuler \cite{a} found a characterization of $6$-regular toroidal  graphs and 
 Negami \cite{n} characterized $6$-regular graphs which embed in the Klein bottle.
Moreover,  no $6$-regular graph embeds in both the torus and the Klein bottle \cite{ln}. 
The inequality  $b(G) \leq 9$ for toroidal graphs,  was proven by Hou and Liu \cite{JHGL}.   
 They also showed that there exist toroidal  graphs with bondage number $7$. 
The next result immediately follows by Theorem \ref{CaDe}.
\begin{proposition}  \label{pro7}
Let $H$ be a 6-regular triangulation in $\mathbb{M}\in \{\mathbb{S}_1, \mathbb{N}_2\}$. 
Then $b(H \circ K_1) =7$.
\end{proposition}
By Corollary \ref{S1N2} and Proposition \ref{pro7} it immediately follows that the maximum value of the bondage number 
 of graph embeddable on surface with Euler characteristic $0$ is $7, 8$ or $9$. 
 The following question naturally arises. 
\begin{question}  \label{q2}
Is there a toroidal graph $G$ with  $b(G) \in \{8,9\}$? 
Is there a Klein bottle  graph $G$ with  $b(G)\in \{8,9\}$? 
\end{question}

\begin{proposition}\label{3d2}
Let $G$ be a connected  toroidal or Klein bottle graph 
and let $\mu \in \{b, b_2\}$.  
\begin{itemize}
\item[(i)] If $\mu (G) > \frac{3}{2}\Delta (G)$ then either 
$4 \leq \delta (G) \leq \Delta (G) \leq 5$ or $G$ is $3$-regular.
\item[(ii)] If $\mu (G) = \frac{3}{2}\Delta (G)$ then 
either 
$G$ is $6$-regular and no edge of $G$ belongs to at least $3$ triangles 
or 
$3 \leq \delta (G) \leq \Delta (G) =4$. 
\end{itemize}
\end{proposition}

\begin{problem}
Find $\max\{b(G)\mid G \ \mbox{is a $6$-regular triangulation in}\ \mathbb{M} \in \{\mathbb{S}_1, \mathbb{N}_2\}$
and no edge of $G$ belongs to at least $3$ triangles\}.   
Find $\max\{b(G)\mid G\  \mbox{is a $4$-regu-}$
\\
\mbox{lar graph embeddable in}\ $\mathbb{M} \in \{\mathbb{S}_1, \mathbb{N}_2\}\}$.
\end{problem}

 For any  graph $G$, which is embeddable  in $\mathbb{N}_3$,
  Gagarin and Zverovich \cite{GagarinZverovich2} 
proved $b(G) \leq 14$. We improve this bound in the following corollary. 
\begin{corollary} \label{N3}
Let $G$ be a graph  embeddable in $\mathbb{N}_3$. 
Then $b(G) \leq B^{\prime} (G) \leq 10$. 
If $G$ has no degree $6$ vertices then $b(G) \leq B^{\prime}(G) \leq 9$. 
If $G$ has $15$ mutually nonadjacent degree $5$ vertices, 
then  $b(G) \leq B^{\prime} (G) \leq 8$. 
\end{corollary}
Since the non-orientable genus of $K_7$ is  $3$  \cite{Ringel}, 
 by Theorem \ref{CaDe} we obtain:
\begin{proposition}  \label{N33}
There  exist  graphs embeddable on $\mathbb{N}_3$ with  bondage number $7$. 
One of them is  $K_7 \circ K_1$.
\end{proposition}
\begin{question}  \label{q3}
Is there a  graph $G$ embeddable in $\mathbb{N}_3$ with  $b(G)  \in \{8, 9,10\}$? 
\end{question}
We conclude our results in this section with a constant upper bound 
 on the bondage number of graphs  embeddable in 
$\mathbb{M} \in \{\mathbb{S}_2, \mathbb{N}_4\}$. For any such a graph $G$, 
$b(G) \leq 16$ (Gagarin and Zverovich \cite{GagarinZverovich2}). We improve 
this result as follows.
\begin{corollary}\label{S2N4}
Let $G$ be a graph  embeddable in $\mathbb{M} \in \{\mathbb{S}_2, \mathbb{N}_4\}$.  
Then $b(G) \leq 12$. 
\end{corollary} 
Since  $h(K_8) = 2$ and $q(K_8) = 4$ \cite{Ringel}, 
 by Theorem \ref{CaDe} we obtain:
\begin{proposition}  \label{S2N41}
There  exist  graphs embeddable on $\mathbb{N}_4$ with  bondage number $8$. 
There  exist  graphs embeddable on $\mathbb{S}_2$ with  bondage number $8$. 
One such a graph is $K_8 \circ K_1$.
\end{proposition}
\begin{question}  \label{q4}
Is there a  graph $G$ embeddable in $\mathbb{N}_4$ with  $b(G)  \in \{9,10,11,12\}$?
Is there a  graph $G$ embeddable in $\mathbb{S}_2$ with  $b(G)  \in \{9,10,11,12\}$? 
\end{question}

\subsection{Proofs}
\begin{lemma} \label{pomost}
Let $G$ be a connected graph $2$-cell embedded  in a surface $\mathbb{U}$.
 Let $s \geq 3$, $V_{\leq s} - V_{\leq 2} \not = \emptyset$ and $B^{\prime}(G) \geq 2s-1$. 
 Let $I = \{x_1, x_2, \dots,x_k\}$ be an independent dominating set in 
$\left\langle  V_{\leq s}, G \right\rangle$. Then $V_{\leq s-1} \subseteq I$ and there is a supergraph  $G_k$ for $G$ 
which is $2$-cell embedded  in  $\mathbb{U}$, $V(G_k) = V(G)$, $E(G) \subseteq E(G_k)$  and the following hold: 
\begin{itemize}
\item [(a)]  $I$ is an independent set of $G_k$;
\item [(b)] if $u \in (V(G) - N_G(I)) \cup N_G(V_1)$ then $N_{G_k} (u) = N_G (u)$;              
\item [(c)] if $u \in I$, $v \in V_{\leq s-1}(G)$ and $u \not = v$ then $d_{G_k}(u,v) = d_G(u,v) \geq 3$;  
\item [(d)]  if $u \in I$, $d_G(u) = r \geq 3$  and $v \in N_G(u)$ then $d_{G_k}(v) \geq 2s -r +2$;
\item [(e)]  if $u \in I$, $d_G(u) = 2$  and $v \in N_G(u)$ then $d_{G_k}(v) \geq 2s-1$.
\end{itemize}
\end{lemma}
\begin{proof}
Since $B^{\prime}(G) \geq 2s-1$,  the following claim is valid.

{\bf Claim 1.}  If $x \in V_{r}(G)$, $r \leq s$, $y \in V(G)$ and 
$1 \leq d_G(x,y) \leq 2$, then $d_G(y) \geq 2s - r$. 

 Hence $V_{\leq s-1} \subseteq I$ and $d_G(x,y) \geq 3$ whenever $x \not = y$, 
$x \in V_{\leq s}$ and $y \in V_{\leq s-1}$. 
Since $G$ is $2$-cell embedded, using Lemma \ref{2connectednew} 
 consecutively $k$ times we obtain the graphs $G_0 = G, G_1,\dots,G_k$, 
 as follows.
 For $r = 1,2,\dots,k$ let $G_r = G_{r-1} + F_r$,
 where $F_r \subseteq \{xy \mid  x,y \in N_G(x_r), x \not = y, xy \not \in E(G_{r-1})\}$, 
such that  $G_r$ is still a $2$-cell embedded  in $\mathbb{U}$ and 
 (i) if $d_G(x_r) \geq 3$ then $\left\langle N_{G_r} (x_r), G_r\right\rangle$ is hamiltonian, and 
(ii)  if $d_G(x_r)  = 2$ then   $x_r$ belong to a triangle of $G_r$. 
Clearly, if $d_G(x_r) \geq 3$ then $\left\langle N_{G_k} (x_r), G_k\right\rangle$ is hamiltonian  and 
 if $d_G(x_r) = 2$ then   $x_r$ belongs to a triangle of $G_k$, $r = 1,2,\dots,k$.

(a)--(c): The results immediately follow by the very definition 
of the graph $G_k$ and by Claim 1.

(d): By Claim 1, $d_G(v) \geq 2s - r$. 
If the equality holds then $N_G(u) \cap N_G(v)$ is empty. 
Since $|N_{G_k}(u) \cap N_{G_k}(v)| \geq 2$, $d_{G_k}(v) \geq 2s-r+2$.
If $d_G(v)= 2s-r+1$ then $|N_{G}(u) \cap N_{G}(v)|=1$.  
Since $|N_{G_k}(u) \cap N_{G_k}(v)| \geq 2$,  $d_{G_k}(v) \geq 2s-r+2$.

(e): By Claim 1, $d_G(v) \geq 2s - 2$. 
If the equality holds then $N_G(u) \cap N_G(v)$ is empty. 
Since $|N_{G_k}(u) \cap N_{G_k}(v)| = 1$, $d_{G_k}(v) \geq 2s-1$.
\end{proof}

\begin{proof}[Proof of Theorem  \ref{bBprim}] 
 Let $G$ be a connected graph $2$-cell embedded  in a surface $\mathbb{M}$ with $\chi (\mathbb{M}) = \chi$.  
Suppose $B^{\prime}(G) \geq 2s-1$. 
Keeping the notation of Lemma \ref{pomost} let us consider the graph 
 $H = G_k - V_{\leq s-1}$. By Clam 1 and Lemma \ref{pomost} we immediately have: 

{\bf Claim 2.}\label{cl2} Let $I_s = I - V_{\leq s-1}$. 
\begin{itemize}
\item [(a)]  $\delta (H) = s$, $I_s = V_s(H)$ and $I_s$ is an independent set of $H$.
\item [(b)] If $u \in V(H) - N_G[I]$ then $N_H(u) = N_G(u)$ and $d_H(u) \geq s+1$. 
\item [(c)]  If $v \in N_G(V_{\leq 2})$ then $d_H (v) \geq 2s-2$.
\item [(d)]  If $v \in N_G(V_l)$, $3 \leq l \leq s-1$, then   $d_{H}(v) \geq 2s -l +1$.
\item [(e)]  If $v \in N_G(I_s)$, then $d_{H}(v) \geq s+2$.
\end{itemize} 

By Lemma \ref{edge} and Claim 2 it follows that 
\begin{eqnarray}%{lll}
\nonumber
6(|H| - \chi) &\geq & 2\|H\|  =  \displaystyle{\sum_{v \in I_s}d_{H}(v) + \sum _{u \in N_{G}(I)}d_{H}(u) + 
  \sum_{t\in V(H)-N_{G}[I]}\hspace{-.8cm}d_{H}(t)} \hfill \\
\nonumber & \geq & \displaystyle{s|I_s|} + ((s+2)|N_G(I_s)| + \displaystyle{|V_{\leq 2}|(2s-2) + \sum_{j=3}^{s-1}(2s-j+1)|V_j|)} \hfill\\
\nonumber & \  & + \displaystyle{(s+1)(|H| - |I_s| - |N_{H}(I_s)| - |N_G(V_{\leq s-1})|)}
\end{eqnarray}
or equivalently 
\begin{eqnarray}\label{qwaa}
 \ & & -6\chi  \geq 
 - |I_s| + |N_{H}(I_s)| + (s-5)|H| + |V_{\leq 2}|(s-3) + \sum_{j=3}^{s-1}(s-j)|V_j|.
\end{eqnarray}

Let us consider the bipartite graph $R$ with parts $I_s$ and $N_H(I_s)$, 
and edge set  $\{uv \in E(G)\mid u \in I_s, v \in N_G(I_s)\}$. 
First let $R$ have a cycle. 
Lemma \ref{edge} implies  $s|I_s| = \|R\| \leq 2(|R| - \chi)$. 
Since $|R| = |I_s| + |N_H(I_s)|$, we obtain 
\begin{equation} \label{nis1}
|N_H(I_s)| \geq \frac{s-2}{2}|I_s| +  \chi
\end{equation}
If $R$ is a forest then $s|I_s| = \|R\| \leq |R| -1 = |I_s| + |N_H(I_s)| -1$. 
Hence $|N_H(I_s)| \geq (s-1)|I_s| +1 \geq \frac{s-2}{2}|I_s| + \chi$. 

By \eqref{qwaa} and \eqref{nis1} it follows
\begin{equation}\label{eq:qwert1}
%\[
 -14\chi  \geq (s-4)|I_s| + 2(s-5)|H| + 2|V_{\leq 2}|(s-3) + 2\sum_{j=3}^{s-1}(s-j)|V_j|.
% \]
\end{equation}
Since $|H| = |G| - |V_{\leq s-1}|$, we finally obtain 
\begin{equation}\label{eq:qwert11}
 -14\chi  \geq (s-4)|I_s| + 2(s-5)|G| + 4|V_{\leq 2}| + 2\sum_{j=3}^{s-1}(5-j)|V_j|,
\end{equation}
 a contradiction. 
\end{proof}

\begin{proof}[Proof of Theorem \ref{deltamax}] 
(a) Since $\chi (\mathbb{M})$ is as large as possible, G has 2-cell embedding on $\mathbb{M}$. 
 Since $G$ has no vertex of degree $s = \delta_{max}^{\mathbb{M}}$, $V_{\leq s-1}$ is not empty.  
Suppose to the contrary that $B^{\prime}(G) \geq 2s -2$.  
 Hence, for any two distinct vertices $x, y \in V_{\leq s-1} = \{x_1,\dots,x_k\}$, $d_G (x,y) \geq 3$. 
 Now, as in the proof of Lemma \ref{pomost}, we obtain a supergraph $G_k$ for $G$ 
 with $V(G) = V(G_k)$ and $xy \in E(G_k) - E(G)$ implies 
 both $x$ and $y$ are in $N_G(u)$ for some $u \in V_{\leq s} (G)$.
 Moreover, 
   if $d_G(x_r) \geq 3$ then $\left\langle N_{G_k}(x_r), G_k \right\rangle$  is hamiltonian, and 
  if $d_G(x_r) = 2$ then  $x_r$ belongs to a triangle of $G_k$, $r = 1,2,\dots,k$.

{\bf Claim 3.}\label{cl3}

(i) If $u \in V_r(G)$, $3 \leq r \leq s-1$ and $v \in N_G(u)$ then $d_{G_k}(v) \geq 2s-r+1$.

(ii) If $u \in V_{\leq 2}(G)$ and $v \in N_G(u)$ then $d_{G_k}(v) \geq 2s-2$.

\begin{proof}[Proof of Claim 3] 
(i): Since $B^{\prime}(G) \geq 2s -2$, $d_G(v) \geq 2s - r -1$.  
 If the equality holds then $N_G(u) \cap N_G(v)$ is empty. 
Since $|N_{G_k}(u) \cap N_{G_k}(v)| \geq 2$, $d_{G_k}(v) \geq 2s-r+1$. 
If $d_G(v)= 2s-r$ then $|N_{G}(u) \cap N_{G}(v)|=1$. 
 Since $|N_{G_k}(u) \cap N_{G_k}(v)| \geq 2$,  $d_{G_k}(v) \geq 2s-r+1$.

(ii): Since $B^{\prime}(G) \geq 2s -2$, $d_G(v) \geq 2s - d_G(u) -1$.  
 If $d_G (u) = 2$ and the equality holds then $N_G(u) \cap N_G(v)$ is empty. 
Since $|N_{G_k}(u) \cap N_{G_k}(v)| = 1$,  $d_{G_k}(v) \geq 2s-2$.
\end{proof}
Consider the graph  $H = G_k - V_{\leq s}(G)$ which is embedded in $\mathbb{M}$. % without crossings. 
Since $s \geq 5$, by Claim 3 it follows $\delta (H) \geq s+1$ - a contradiction.

(b) The result immediately follows by (a) and Lemma \ref{fivedegree}. 
\end{proof}

\begin{proof}[Proof of Corollary \ref{S1N2}] 
If $\delta (G) \geq 6$ then $G$ is a 6-regular triangulation 
as it follows by the Euler formula; hence $B^{\prime}(G)=9$. 
If $V_5(G)$ is not empty then $B^{\prime}(G) \leq 8$ by Corollary \ref{pet}.
So, let $V_{\leq 4} (G) \not = \emptyset$ and $V_5(G) = \emptyset$. 
Suppose  $B^{\prime}(G) \geq 9$. %Then  $\Delta (G) \geq 6$. 
Note that if $x \in V_{r}(G)$, $r \leq 4$, $y \in V(G)$ and 
$1 \leq d_G(x,y) \leq 2$, then $d_G(y) \geq 10 - r$. 
 Hence $V_3(G) \cup V_4(G) \not = \emptyset$ - 
 otherwise each component of the graph $G-V_{\leq 2}(G)$
 is a  graph with minimum degree at least $6$ and maximum degree at least $7$, 
 contradicting Lemma \ref{edge}.  
 Consider the supergraph $G_k$ of $G$ described  in Lemma \ref{pomost}, provided $s=5$.  
Then  Lemma \ref{pomost} implies  the graph $H = G_k - V_{\leq 4}$ has minimum degree at least $6$ 
and maximum degree at least $7$ -  
 again a contradiction with Lemma \ref{edge}. 
\end{proof}

\begin{proof}[Proof of Proposition \ref{3d2}] 
By Theorem \ref{samczech} and Theorem \ref{d2} it follows that 
$\mu (G) \leq \Delta (G) + 3$. Since $G$ is $2$-cell embedded, 
$\Delta (G) \geq 3$. 

(i) Since $\mu (G) > \frac{3}{2}\Delta (G)$, $\Delta (G) \leq 5$. 
Assume $\delta (G) \leq 3$. But then $b_2 (G) \geq \mu (G)$ implies that  
$G$ is $3$-regular. 

(ii) Since $\mu (G) = \frac{3}{2}\Delta (G)$, $\Delta (G) \in \{4,6\}$. 
 If $\Delta (G) = 6$ then $b_2 (G) \geq \mu (G) = 9 = \Delta (G) + 3$. 
 By Theorem \ref{samczech}, 
$G$ is $6$-regular and no edge of $G$ belongs to at least $3$ triangles. 
 So, let $\Delta (G) =4$. Then $\mu (G) =6$ which leads to $\delta (G) \geq 3$. 
\end{proof}

\begin{proof}[Proof of Corollary \ref{N3}] 
If $G$ is embeddable in a surface with non-negative Euler characteristic then 
the result follows by Corollary \ref{S0N1} and Corollary \ref{S1N2}. 
 So, we may assume that the non-orientable genus of $G$ is $3$ and hence $|G| \geq 7$. 
By Lemma \ref{edge}, $\|G\| \leq 3|G| + 3$. 
Hence $\delta_{max}^{\mathbb{N}_3} = 6$. 
If $G$ has no degree $6$ vertices then $B^{\prime}(G) \leq 9$ because of Theorem \ref{deltamax}. 
Assume $V_6$ is not empty. 
But then 
$ 7 < \beta_0 (\left\langle V_6, G\right\rangle) + |G|$. 
Now by Corollary \ref{pet}(ii),  $b(G)\leq B^{\prime}(G) \leq 10$. 
The rest immediately follows by Corollary \ref{pet}(i). 
\end{proof}

\begin{proof}[Proof of Corollary \ref{S2N4}]
If $G$ is embeddable in a surface with  Euler characteristic not less than $-1$ then 
the result follows by Corollary \ref{S0N1}, Corollary \ref{S1N2} and Corollary \ref{N3}. 
 So, we may assume that  at least one of $q(G) = 4$ and  $h(G) =2$ holds.   
By Lemma \ref{edge}, $\|G\| \leq 3|G| + 6$. 
Hence $\delta_{max}^{\mathbb{M}} \leq 7$. 
 Since $h(K_8)=2$ and $q(K_8)=4$,  $\delta_{max}^{\mathbb{M}}=7$. 
If $G$ has no degree $7$ vertices then $b(G) \leq B^{\prime}(G) \leq 11$
 because of Theorem \ref{deltamax}. 
Assume $V_7$ is not empty. 
If $V_6$ is empty then Corollary \ref{pet}(iii) implies 
$b(G) \leq B^{\prime}(G) \leq 12$. So, let $V_6 \not = \emptyset$. 
 If there are $u \in V_6$ and $v \in V_7$ which are at distance at most $2$
  then $b(G) \leq B^{\prime}(G) \leq b_1(G) \leq 6+7-1 = 12$. 
  If  $u \in V_6$, $v \in V_7$ and $d_G(u,v) \geq 3$ then $|G| \geq 15$.  
  By Corollary \ref{pet}(ii), $b(G) \leq B^{\prime}(G) \leq 10$. 
\end{proof}

\section{Upper bounds: the domination number}
In this section we present upper bounds for the order and bondage number of a graph 
in terms of the domination number and Euler characteristic. 
To do this we need the following results.
\begin{them} \label{san} {\rm (Sanchis~\cite{sanchis})}
Let $G$ be a connected graph with $n$ vertices, domination number $\gamma$ where
$3 \leq \gamma \leq  n/2$. Then the number of edges of $G$ is at most
$(n-\gamma +1)(n-\gamma)/2$. If $G$ has exactly this number of edges and $\gamma \geq 4$ 
it must be of the following form.
\begin{itemize}
\item[$(P_1)$]  An $(n - \gamma)$-clique, together with an independent set of size $\gamma$, 
such that each of the vertices in the $(n - \gamma)$-clique is adjacent to exactly one of 
the vertices in the independent set, and such that each of these $\gamma$ vertices
 has at least one vertex adjacent to it.
 \item[$(P_2)$]  For $\gamma = 3$, $G$ may consist of a clique of $n-5$ vertices, 
 together with $5$ vertices $x_1, x_2, x_3, x_4, x_5$, with edges $x_1x_3$, $x_2x_4$, $x_2x_5$, 
  such that every vertex in the $(n-5)$-clique is adjacent to 
  $x_4$ and $x_5$, and in addition adjacent to either $x_l$ or $x_3$. 
Moreover, at least one of these vertices is adjacent to $x_l$ and at least one to $x_3$. 
\end{itemize}
\end{them}

\begin{them} \label{orethm} {\rm (Ore~\cite{ore})}
If $G$ is a connected graph with $n \geq 2$ vertices then $\gamma (G) \leq n/2$. 
\end{them}

\begin{proposition} \label{upper}
Let $G$ be a connected graph of order $n\geq 2$ which is $2$-cell embedded  in a surface $\mathbb{M}$. 
\begin{itemize}
%\item[(i)]  If $\gamma (G) =1$ then $n \geq (3+\sqrt{17-8\chi(\mathbb{M}})/2$. 
\item[(i)] If $\gamma (G) =2$ then  $n \geq 2 + \sqrt{6-2\chi(\mathbb{M})}$ when 
 $n$ is even and $n \geq 2 + \sqrt{7-2\chi(\mathbb{M})}$ when $n$ is odd.
\item[(ii)] If $\gamma (G) = \gamma \not = 2$ then 
\begin{equation} \label{nga}
n \geq \gamma + (1 + \sqrt{9+8\gamma-8\chi(\mathbb{M})})/2, \  \mbox{and}
\end{equation}
\begin{equation} \label{gan}
\gamma \leq n + (1 - \sqrt{8n + 9 - 8\chi(\mathbb{M})})/2.
\end{equation}
\end{itemize}
\end{proposition}
\begin{proof}[Proof of Proposition \ref{upper}]
Since $f(G) \geq 1$, Euler's formula implies  $n - \|G\| + 1 \leq \chi(\mathbb{M})$.

(i) If $H$ is a graph with $\gamma (H) =2$, $|H| = n$  and  maximum number of edges 
then its complement is a forest in which each component is a star \cite{sb}.
This implies $n(n-1)/2 - \left\lceil n/2 \right\rceil = \|H\| \geq \|G\|$.  
 Hence $n - n(n-1)/2 + \left\lceil n/2 \right\rceil + 1 \leq \chi(\mathbb{M})$. 
      Equivalently, $n^2 - 4n + 2\chi(\mathbb{M}) - 2 \geq 0$ when $n$ is even and 
      $n^2 - 4n + 2\chi(\mathbb{M}) - 3 \geq 0$ when $n$ is odd. 
      Since $n \geq 2$, the result easily follows. 

(ii) Since $\|G\| \leq (n-\gamma +1)(n-\gamma)/2$ (by Theorem \ref{san} when $\gamma \geq 3$), 
we have $2\chi(\mathbb{M}) \geq 2n - (n-\gamma +1)(n-\gamma) +2$, or equivalently
\[
n^2 - (2\gamma+1)n + \gamma^2 - \gamma - 2 + 2\chi(\mathbb{M}) \geq 0 \ \mbox{and} 
\]
\[
\gamma^2 - (2n+1)\gamma + n^2 - n - 2 + 2\chi(\mathbb{M}) \geq 0. 
\]
Solving these inequalities we respectively obtain 
 \eqref{nga} and \eqref{gan}, because 
$n \geq 2\gamma$ (by Theorem \ref{orethm}). 
\end{proof}

Next we show that the bounds in Proposition \ref{upper}(ii) are tight. 
Let a graph $G$ have property $(P_1)$(Theorem \ref{san})
and in addition $\delta (G) \geq 4$, $|G| = n = \gamma +i+4t$, where 
 $t \geq \gamma = \gamma (G) \geq 4$,  
 $i = 1$ when $\gamma$ is odd, and  $i = 2$ when $\gamma$ is even. 
 If $p = (\|G\| - |G| +1)/2$ then $p = 4t^2+t + (1-\gamma)/2$ when $\gamma$ is odd, 
 and $p = 4t^2 +3t +1 - \gamma/2$ when $\gamma$ is even. 
  Since $G$ is clearly $4$-edge connected, 
  $G$ can be embedded in $\mathbb{M} = \mathbb{S}_p$(e.g. see Jungerman \cite{Jung}). 
  Note also that $G$  can be $2$-cell embedded in $\mathbb{N}_{2p}$(see \cite{Ringel}).   
    It is easy to see that, in both cases, we have equalities in \eqref{nga} and  \eqref{gan}.

 Combining Theorem\ref{SGZ}(i) and Proposition \ref{upper} we immediately
  obtain the following results on the average degree of a graph.
\begin{corollary} \label{aver}
Let $G$ be a connected graph  $2$-cell embedded  in a surface $\mathbb{M}$ 
 with  $\chi(\mathbb{M}) = \chi \leq -1$. 
  \begin{itemize}
   \item[(i)] Then $ad(G)\leq 6 - 12\chi/(3+\sqrt{17-8\chi})$. 
   \item[(ii)] If  $\gamma (G) = 2$ then $ad(G)\leq 6 - 6\chi/(2+\sqrt{6-2\chi})$ 
   when $|G|$ is even, and $ad(G)\leq 6 - 6\chi/(2+\sqrt{7-2\chi})$ when $|G|$ is odd. 
    \item[(iii)] If  $\gamma (G) = \gamma \geq 3$ and $g(G) = g$ then 
\[
ad (G) \leq \frac{2g}{g-2}(1 - \frac{2\chi}{2\gamma + 1 + \sqrt{9 + 8\gamma - 8\chi}})
\]
\[
\leq 6 - \frac{12\chi}{2\gamma + 1 + \sqrt{9 + 8\gamma - 8\chi}}
\leq 6 - \frac{12\chi}{7 + \sqrt{33 - 8\chi}}.
\]
    \end{itemize}
\end{corollary}
The next theorem follows by combining Theorem\hspace{-.1cm}\ref{SGZ}(ii) and
 Corollary \ref{aver}.
\begin{theorem}\label{dom}
Let $G$ be a connected graph  $2$-cell embedded  in a surface $\mathbb{M}$ 
 with  $\chi(\mathbb{M}) = \chi \leq -1$.  
 \begin{itemize}
 \item[(i)] If  $\gamma (G) = 2$ then 
 \[
 b(G) \leq 2ad (G) - 1 \leq 11 - \frac{12\chi}{2+\sqrt{6-2\chi}} \ \mbox{when $|G|$ is even, and}
 \]
 \[
 b(G) \leq 2ad (G) - 1 \leq 11 - \frac{12\chi}{2+\sqrt{7-2\chi}} \  \mbox{when $|G|$ is odd.}
 \]
 \item[(ii)] If  $\gamma (G) = \gamma \geq 3$ and $g(G) = g$ then 
\[
b(G) \leq 2ad (G) - 1 \leq 3 + \frac{8}{g-2} - \frac{8g}{g-2}.\frac{\chi}{2\gamma + 1 + \sqrt{9 + 8\gamma - 8\chi}}
\]
\[ 
\leq 11 - \frac{24\chi}{2\gamma + 1 + \sqrt{9 + 8\gamma - 8\chi}} 
\leq 11 - \frac{24\chi}{7 + \sqrt{33 - 8\chi}}.
\]
\end{itemize}
\end{theorem}
Let us note that the bounds stated in Theorem \ref{dom} are better than the 
one in Theorem \ref{GZ11} whenever  $\gamma (G) \geq 2$. 
Finding  a better upper bound for $b(G)$ than the bound stated in 
Theorem \ref{dom}(ii) could help answer the following question.
\begin{question}
 What is  the maximum number of edges in a connected graph of order $n$, domination number $\gamma$ 
 and girth $g$, where $1 \leq \gamma \leq n/2$ and $g\geq 4$.
\end{question}

\section{Remarks}

Teschner \cite{Teschner}  proved that Conjecture \ref{con1} holds when 
the domination number of  a graph $G$ is not more than $3$. 
\begin{them}[Teschner ~\cite{Teschner}]\label{tend}
Let $G$ be a connected graph.  
\begin{itemize}
\item[(i)] If $\gamma(G) = 1$ then $b(G) = \left\lceil \frac{t}{2}\right\rceil 
\leq \frac{1}{2}\Delta(G) + 1 \leq \frac{3}{2}\Delta(G)$, 
where $t$ is the number of vertices of degree $|G|-1$.
\item[(ii)] If $\gamma(G) = 2$ then $b(G) \leq \Delta(G) + 1 \leq \frac{3}{2}\Delta(G)$.
\item[(iii)] If $\gamma(G) =3$ then $b(G) \leq \frac{3}{2}\Delta(G)$.
\end{itemize}
\end{them}

Hence it is naturally to turn our attention toward the graphs 
with the domination number at least $4$. 
 By Theorem \ref{dom}(ii)  we have 
\[
b(G) \leq 2ad (G) - 1 \leq 
11 - \frac{24\chi}{9 + \sqrt{41 - 8\chi}}\]
whenever 
 $G$ is a connected graph  $2$-cell embedded  in a surface $\mathbb{M}$,  
 $\chi(\mathbb{M}) = \chi \leq -1$  and $\gamma (G) \geq 4$.
 For a graph $G$ which has $2$-cell embedding on a surface 
 with Euler characteristic $\chi \in \{-2,-3,\dots,-23\}$, we have the
upper bounds shown in Table 1 provided $\gamma (G) \geq 4$.

\begin{table}[h!] 
\centerline {\small%\footnotesize %\scriptsize
	\begin{tabular}[t]{|| r || r | r | r | r | r | r | r | r | r | r | r | r | ||}
	\hline
	Euler characteristic, $\chi$ & -2 &-3 &-4 & -5 & -6 & -7& -8 & -9 & -10 & -11 & -12\\
	\hline
	$b(G) \leq 2ad (G) - 1 \leq $  & 13 & 15 & 16 & 17 & 18 & 19 & 20 & 22 & 23 & 23 & 24  \\
	\hline
	\hline
		Euler characteristic, $\chi$ & -13& -14 & -15& -16& -17& -18 & -19 & -20 &-21 & -22 & -23 \\
	\hline
	$b(G) \leq 2ad (G) - 1 \leq $  & 25 & 26 & 27 & 28 & 29 & 30 & 30 & 31 & 32 & 33 & 34  \\
	\hline
	\end{tabular} }
	\vspace{.2cm}
	
\caption{\small Constant upper bounds for the bondage number of graphs: $\gamma \geq 4$ and $\chi \in \{-2,-3,\dots,-23\}$ .\label{tab-chi1}}
\end{table}
For the sake of completeness we add the upper bounds presented in section \ref{sver}.
\begin{table}[h!] 
\centerline {\small%\footnotesize %\scriptsize
	\begin{tabular}[t]{|| r || r | r | r | r | r | }
	\hline
	Euler characteristic, $\chi$ & 2 & 1 & 0 & -1 & -2 \\
	\hline
	$b(G)\leq B^{\prime}(G) \leq $  & 8 & 8 & 9 & 10 & 12  \\
	\hline
	\end{tabular} }
	\vspace{.2cm}
	
\caption{\small  Constant upper bounds for the bondage number of graphs: $\chi \geq -2$ .\label{tab-chi2}}
\end{table}

Recall  that the only known connected graphs for which the equality in Teschner's conjecture holds 
 are $K_n \times K_n$, $n \geq 2$.  We conclude by:
\begin{question}
Is there a connected graph $G$ such that $G \not = K_n \times K_n$ and $b(G) = \frac{3}{2}\Delta (G)$?
\end{question}

\end{document}